\documentclass[oneside]{amsart}
\def\notes{0}
\usepackage{amssymb, latexsym}
\usepackage{enumerate}
\usepackage{tikz}
\usetikzlibrary{arrows,matrix}
\usepackage{float}
\usepackage{tabularx}
\usepackage{color}
\usepackage[numbers]{natbib}
\usepackage{hyperref}

\ifnum\notes=1
\usepackage[width=6in,right=1.9in,marginpar=1.6in,height=9in]{geometry}
\else
\usepackage[margin=1.3in]{geometry}
\fi

\newtheorem{thm}{Theorem}
\newtheorem{lem}[thm]{Lemma}

\newtheorem{prop}[thm]{Proposition}

\newtheorem{mydef}[thm]{Definition}

\theoremstyle{definition}

\theoremstyle{definition}

\theoremstyle{definition}

\ifnum\notes=1
\newcommand{\mynote}[1]{\marginpar{\tiny\sf {#1}}}
\else
\newcommand{\mynote}[1]{}
\fi

\usepackage{accents}
\newlength{\dhatheight}
\newcommand{\doublehat}[1]{%
    \settoheight{\dhatheight}{\ensuremath{\hat{#1}}}%
    \addtolength{\dhatheight}{-0.2ex}%
    \hat{\vphantom{\rule{1pt}{\dhatheight}}%
    \smash{\hat{#1}}}}

\newcommand{\defeq}{\overset{\text{def}}{=}}

\begin{document}

\title{When is Nontrivial Estimation Possible\\ for Graphons and Stochastic Block Models?}
\author{Audra McMillan \and Adam Smith}
\maketitle

\begin{abstract}
Block graphons (also called stochastic block models) are an important and widely-studied class of models for random networks. 
We provide a lower bound on the accuracy of estimators for block graphons with a large number of blocks. We show that, given only the number $k$ of blocks and an upper bound $\rho$ on the values (connection probabilities) of the graphon, every estimator incurs error $\Omega\left(\min\left(\rho, \sqrt{\frac{\rho k^2}{n^2}}\right)\right)$ in the $\delta_2$ metric with constant probability for at least some graphons. In particular, our bound rules out any nontrivial estimation (that is, with $\delta_2$ error substantially less than $\rho$) when $k\geq n\sqrt{\rho}$. Combined with previous upper and lower bounds, our results characterize, up to logarithmic terms, the accuracy of graphon estimation in the $\delta_2$ metric.  A similar lower bound to ours was obtained independently by \citet*{Klopp:2016}.  
\end{abstract}

\section{Introduction}

Networks and graphs arise as natural modeling tools in many areas of science. In many settings, especially ones where edges in a network represent social ties, observed networks display some type of community structure, where the connectivity between nodes depends heavily on the communities they belong to. This type of structure is captured in the \emph{$k$-block graphon} model, also known as the stochastic block models. The more communities we allow in the model (or ``types" of nodes we consider), the richer the model becomes and the better we can hope to describe the real world. 

Given an observed network, graphon estimation is the problem of finding a graphon model that approximates the process that gave rise to the network. In this paper, we are concerned with the fundamental limits of graphon estimation for block graphons. That is, given a $n$-node network that was generated from a $k$-block graphon, how accurately can you recover the graphon? We consider the ``nonparametric'' setting, where $k$ may depend on $n$. Our lower bounds apply even to estimation algorithms that know the true number of blocks $k$ (though this quantity typically needs to be estimated).


Many real world networks display the property that the average degree of the network is small compared to the number of nodes in the network. Graphons whose expected average degree is linear in $n$ are called dense, while graphons whose expected average degree is sublinear in $n$ are referred to as sparse. In this work, we prove a new lower bound for graphon estimation for sparse networks. In particular, our results rule out \emph{nontrivial} estimation for very sparse networks (roughly, where $\rho = O(k^2/n^2)$). An estimator is nontrivial if its expected error is significantly better than an estimator which ignores the input and always outputs the same model. It follows from recent work \cite{Mossel:2014,Mossel:2015,Neeman:2014} that nontrivial estimation is impossible when $\rho=O(1/n)$.
Ours is the first lower bound that rules out nontrivial graphon estimation for large $k$. 
Previous work by \citet*{Klopp:2015} provides other lower bounds on graphon estimation that are tight in several regimes. In very recent work~\cite{Klopp:2016} that is concurrent to ours, the same authors provide a similar bound to the one presented here.


Block graphon models were introduced by \citet{HRH02} under the name latent position graphs. Graphons play an important role in the theory of graph limits (see \cite{Lovasz:2012} for a survey) and the connection between the graph model and convergent graph sequences has been studied in both the dense and sparse settings \cite{BCLSV06, BCLSV08, BCCZ14b, BCCZ14a, BCCZ14b}. Estimation for stochastic block models with a fixed number of blocks was introduced by \citet{BC09}, while the first estimation of the general model was proposed by \citet*{BCL11}. Many graphon estimation methods, with an array of assumptions on the graphon, have been proposed since; \cite{LOGR12, TSP13, LR13, WO13, CA14, ACC13, YangHA14, Gao:2014, ABH14, Chatterjee15, AS15known}. \citet{Gao:2014} provide the best known upper bounds in the dense setting while \citet{WO13,BorgsCS15,Klopp:2015} give upper bounds for the sparse case. 

\subsection{Graphons} 
\begin{mydef}[Bounded Graphons and $W$-random graphs]
A (bounded) \emph{graphon} $W$ is a symmetric, measurable function $W:[0,1]^2\to [0,1]$. Here, symmetric means that $W(x,y)=W(y,x)$ for all $(x,y)\in [0,1]^2$.

For any integer $n$, a graphon $W$ defines a distribution on graphs on $n$ vertices as follows: First, select $n$ labels $\ell_1, \cdots, \ell_n$ uniformly and independently from $[0,1]$, and form an $n\times n$ matrix $H$ where $H_{ij} = W(\ell_i, \ell_j)$. We obtain an unlabeled, undirected graph $G$ by connecting the $i$th and $j$th nodes with probability  $H_{ij}$ independently for each $(i,j)$. The resulting random variable is called a \emph{$W$-random graph}, and denoted $G_n(W)$. 

For $\rho \geq 0$, we say a graphon is $\rho$-bounded if $W$ takes values in $[0,\rho]$ (that is, $\|W\|_\infty \leq \rho$). 
\end{mydef}

We denote the set of graphs with $n$ nodes by $\mathcal{G}_n$, the set of graphons by $\mathcal{W}$ and the set of $\rho$-bounded graphons by $\mathcal{W}_\rho$. If $W$ is $\rho$-bounded, then the expected number of edges in $G_n(W)$ is at most  $\rho\binom n 2 = O(\rho n^2)$. In the case that $\rho$ is parametrised by $n$ and $\lim_{n\to\infty}\rho\to0$, we obtain a sparse graphon.

We consider the estimation problem: given parameters $n$ and $\rho$, as well as a graph $G\sim G_n(W)$ generated from an unknown $\rho$-bounded graphon $W$, how well can we estimate $W$?

A natural goal is to design estimators that produce a graphon $\hat W$ that is close to $W$ in a metric such as $L_2$. This is not possible, since there are many graphons that are far apart in $L_2$, but that generate the same probability distribution on graphs. If 
there exists a measure preserving map $\phi:[0,1]\to[0,1]$ such that $W(\phi(x), \phi(y))=W'(x,y)$ for all $x,y\in[0,1]$, then $G_n(W)$ and $G_n( W')$ are identically distributed. The converse is true if we instead only require $W(\phi(x), \phi(y))=W'(x,y)$ almost everywhere. Thus, we wish to say that $\hat W$ approaches the \emph{class} of graphons that generate $G_n(W)$. To this end, we use the following metric on the set of graphons, $$\delta_2(W,W') = \inf_{{\phi:[0,1]\to[0,1] \atop \text{measure-preserving}}} \|W_{\phi}-W'\|_2$$ where $W_{\phi}(x,y)=W(\phi(x),\phi(y))$ and $\phi$ ranges over all measurable, measure-preserving maps. Two graphons $W$ and $W'$ generate the same probability distribution on the set of graphs if and only if $\delta_2(W, W')=0$ (see, e.g. \cite{Lovasz:2012}).

Existing upper bounds for graphon estimation are based on algorithms that produce graphons of a particular form, namely \emph{block graphons}, also called \emph{stochastic block models} (even when it is not known that the true graphon is a block graphon).

\begin{mydef}[$k$-block graphon (stochastic block models)]
For $k\in\mathbb{N}$, a graphon is a \emph{$k$-block graphon} if there exists a partition of $[0,1]$ into $k$ measurable sets $I_1, \cdots, I_k$ such that $W$ is constant on $I_i\times I_j$ for all $i$ and $j$. 
\end{mydef}

Graphons of this form can be generated from $k\times k$ matrices. Given a $k\times k$ matrix $M$, we can assign a $k$-block graphon $W[M]$ with blocks $I_i=(\frac{i-1}{k}, \frac{i}{k}]$ that takes the value $M_{ij}$ on $I_i\times I_j$.  

\subsection{Main result} 

We are concerned with the problem of estimating a graphon, $W$, given a graph sampled from $G_n(W)$. A graphon estimator is a function $\hat{W}: \mathcal{G}_n \to \mathcal{W}$ that takes as input a $n$ node graph, that is generated according to $W$, and attempts to output a graphon that is close to $W$. The main contribution of this paper is the development of the lower bound $$\inf_{\hat{W}}\sup_{W}\underset{G\sim G_n(W)}{\mathbb{E}}[\delta_2(\hat{W}, W)]\ge\Omega\left(\min\left(\rho, \sqrt{\frac{\rho k^2}{n^2}}\right)\right).$$ Combined with previous work we can give the following lower bound on the error of graphon estimators. 

\begin{thm}\label{main} For any positive integer $2\le k\le n$ and $0< \rho\le 1$, 
$$\inf_{\hat{W}}\sup_{W}\underset{G\sim G_n(W)}{\mathbb{E}}\left[\delta_2(\hat{W}(G), W)\right]\ge \Omega\left(\min\left(\rho, \;\rho\sqrt[4]{\frac{k}{n}}+\sqrt{\frac{\rho k^2}{n^2}} + \sqrt{\frac{\rho \log (\min(k,\rho n+2))}{n}}\right)\right).$$
where $\inf_{\hat{W}}$ is the infimum over all estimators $\hat{W}:G_n\to\mathcal{G}$ and $\sup_W$ is the supremum over all $k$-block, $\rho$-bounded graphons.
\end{thm}

Note that $k$ and $\rho$ may depend on $n$. That is, the theorem holds if we consider sequences $\rho_n$ and $k_n$. Our result improves on previously known results when $\rho=O\left(\left(\frac{k}{n}\right)^{3/2}\right)$ and $n<\frac{k^2}{\log k}$, that is, when the graphs produced by the graphon are sparse and $k$ is large. The upper bound 
\begin{equation}\label{Kloppupper}
\inf_{\hat{W}}\sup_{W}\underset{G\sim G_n(W)}{\mathbb{E}}\left[\delta_2(\hat{W}(G), W)\right]\le O\left(\min\left(\rho, \;\rho\sqrt[4]{\frac{k}{n}}+\sqrt{\frac{\rho k^2}{n^2}}+\sqrt{\frac{\rho \log k}{n}}\right)\right)
\end{equation} 
by Klopp et al. \cite{Klopp:2015} implies that our lower bound is almost tight. In particular, if $k$ is constant then the lower bound in Theorem \ref{main} is tight. In all cases, it is within a factor of at most $\max(\frac{\log k}{\log(\rho n+2)},1)$ of the correct bound.

When $\rho=O\left(\frac{k^2}{n^2}\right)$, Theorem \ref{main} implies that the error is $\Omega(\rho)$, which is the error achieved by the trivial estimator $\hat{W}=0$. That is, in the sparse setting, the trivial estimator achieves the optimal error. To the authors' knowledge this is the first result that completely rules out nontrivial estimation in the case where $k$ is large. Recent concurrent work (\citet{Klopp:2016}) provides similar bounds.

The bound 
\begin{equation}\label{Klopplower}
\inf_{\hat{W}}\sup_{W}\underset{G\sim G_n(W)}{\mathbb{E}}[\delta_2(\hat{W}, W)]\ge \Omega\left(\rho\sqrt[4]{\frac{k}{n}}\right)
\end{equation} is due to previous work of Klopp et al. \cite{Klopp:2015} and the bound 
\begin{equation}\label{Neemanlower}
\inf_{\hat{W}}\sup_{W}\underset{G\sim G_n(W)}{\mathbb{E}}[\delta_2(\hat{W}, W)]\ge\Omega\left( \min\left(\rho, \sqrt{\frac{\rho \log (\min(k,\rho n+2))}{n}}\right)\right)
\end{equation} follows from a result of Neeman and Netrapalli \cite{Neeman:2014}. We give details on how to derive \eqref{Neemanlower} from their results in Section \ref{extra}.  

\subsection{Techniques: Combinatorial Lower Bounds for $\delta_p$} 
Our proof of the main theorem will involve Fano's lemma. As such, during the course of the proof we will need to lower bound the packing number, with respect to $\delta_2$, of a large set of $k$-block graphons. Whilst easily upper bounded, little is known about lower bounds on $\delta_2$. To the authors' knowledge, this work gives the first lower bound for the packing number of $\mathcal{W}_\rho$ with respect to $\delta_2$. We will also give a combinatorial lower bound for the $\delta_2$ metric that is easier to handle than the metric itself.


To understand our technical contributions, it helps to first understand a problem related to graphon estimation, namely that of estimating the matrix of probabilities $H$. Existing algorithms for graphon estimation are generally analyzed in two phases: first, one shows that the estimator $\hat W$ is close to the matrix $H$ (in an appropriate version of the $\delta_2$ metric), and then uses (high probability) bounds on $\delta_2(W,W[H])$ to conclude that $\hat{W}$ is close to $W$. \citet{Klopp:2015} show tight upper and lower bounds on estimation of $H$. One can think of our lower bound as showing that the lower bounds on estimation of $H$ can be transferred to the problem of estimating $W$. 

The main technical difficulty lies in showing that a given pair of matrices $A,B$ lead to graphons that are far apart in the $\delta_2$ metric.  Even if $A,B$ are far apart in, say, $\ell_2$, they may lead to graphons that are close in $\delta_2$. For consistency with the graphon formalism, we normalize the $\ell_2$ metric on $k\times k$ matrices so that it agrees with the $L_2$ metric on the corresponding graphons. For a $k\times k$ matrix $A$,

$$\|A\|_2 \defeq  \Big(\frac 1 {k^2} \sum_{i,j\in [k]} A_{ij}^2\Big)^{1/2} = \|W[A]\|_2.$$
As an example of the discrepancy between the $\ell_2$ and $\delta_2$ metrics, consider the matrices $$A = \begin{pmatrix} 1 & 0 & 1\\ 0&1&0\\1&0&1\end{pmatrix} \hspace{0.5in}\text{and}\hspace{0.5in} B= \begin{pmatrix} 1 & 1 & 0\\ 1&1&0\\0&0&1\end{pmatrix}.$$ The matrices $A$ and $B$ have positive distance in the $\ell_2$ metric, $\|A-B\|_2 = \frac{2}{3}$, but $\delta_2(W[A], W[B])=0$.

One can get an \emph{upper bound} on $\delta_2(W[A],W[B])$ by restricting attention in the definition of $\delta_2$ to functions $\phi$ that permute the blocks $I_i$. This leads to the following metric on $k\times k$ matrices which minimizes over permutations of the rows and columns of one of the matrices:
$$\hat{\delta}_2(A,B) \defeq \min_{\sigma\in\mathcal{S}_k}\|A_{\sigma}-B\|_2\,,$$ 
where $A_\sigma$ is the matrix with entries $(A_\sigma)_{ij} = A_{\sigma(i),\sigma(j)}$. This metric arises in other work (e.g. \cite{Lovasz:2012}), and it is well known that 
$$\delta_2(W[A],W[B]) \leq \hat{\delta}_2(A,B).$$ 

To prove lower bounds, we consider a new metric on matrices, in which we allow the rows and columns to be permuted separately. Specifically, let 
$$\doublehat{{\delta}}_2(A,B) \defeq \min_{\sigma, \tau\in\mathcal{S}_k}\|A_{\sigma, \tau}-B\|_2\, ,$$
where $A_{\sigma,\tau}$ is the $k\times k$ matrix with entries $(A_{\sigma,\tau})_{ij} = A_{\sigma(i),\tau(j)}$. 

\begin{lem}[Lower bound for $\delta_2$]\label{lowerbounddelta2}
  For every two $k\times k$ matrices $A,B$, 
  $$\doublehat\delta_2(A,B) \leq \delta_2(W[A], W[B]).$$ 
\end{lem}

Because $\doublehat \delta_2$ is defined ``combinatorially'' (that is, it involves minimization over a discrete set of size about $2^{2k\ln k}$, instead of over all measure-preserving injections), it is fairly easy to lower bound $\doublehat \delta_2(A,B)$ for random matrices $A,B$ using the union bound.

In particular, it allows us to give bounds on the packing number of $\mathcal{W}_\rho$ with respect to the $\delta_2$ metric. If $(\Omega, d)$ is a metric space, $\epsilon>0$ and $T\subset\Omega$, then we define the $\epsilon$-packing number of $T$ to be the largest number of disjoint balls of radius $\epsilon$ that can fit in $T$, denoted by $\mathcal{M}(\epsilon, T, d)$. The following Proposition will be proved after the proof of Theorem \ref{main}.

\begin{prop}\label{covering}
There exists $C>0$ such that the $C\rho$-packing number of $\mathcal{W}_{\rho}$, equipped with $\delta_2$, is $2^{\Omega(k^2)}$, that is $\mathcal{M}(C\rho, \mathcal{W}_{\rho}, \delta_2) = 2^{\Omega(k^2)}$.
\end{prop}

Finally, we note that these techniques extend directly to the $\delta_p$ metric, for $p \in [1,\infty]$. That is, we may define $\delta_p, \hat\delta_p$ and $\doublehat\delta_p$ analogously to the definitions above, and obtain the bounds
$$\doublehat\delta_p(A,B) \leq \delta_p(W[A], W[B])\leq \hat\delta_p(A,B),$$ along with similar lower bounds on the packing number.



\hfill \subsection{Related work} \hfill

Work on graphon estimation falls broadly into two categories; estimating the matrix $H$ and estimating the graphon $W$. When estimating $H$, the aim is to produce a matrix that is close in the $\ell_2$ metric to the true matrix of probabilities $H$ that was used to generate the graph $G$. When estimating the graphon, our aim is the minimise the $\delta_2$ distance between the estimate and the true underlying graphon $W$ that was used to generate $G$.

Gao et al. studied the problem of estimating the matrix of probabilities $H$ given an instance chosen from $W$ when $\rho=1$. They proved the following minimax rate for this problem when $W$ is a $k$-block graphon; $$\inf_{\hat{M}}\sup_{H}\underset{G\sim G_n(H)}{\mathbb{E}}\left[\frac{1}{n^2}\|\hat{M}(G)-H\|_2\right]\asymp \sqrt{\frac{k^2}{n^2}+\frac{\log k}{n}}$$ where the infinimum is over all estimators $\hat{M}$ from $G_n$ to the set of symmetric $n\times n$ matrices, the supremum is over all probability matrices $H$ generated from $k$-block graphons. Klopp et al. extended this result to the sparse case, proving that for all $k\le n$ and $0<\rho\le 1$, $$\inf_{\hat{M}}\sup_{H}\underset{G\sim G_N(H)}{\mathbb{E}}\left[\frac{1}{n^2}\|\hat{M}(G)-H\|_2\right]\ge\Omega\left(\min\left(\sqrt{\rho\left(\frac{k^2}{n^2}+\frac{\log k}{n}\right)}, \rho\right)\right)$$ where the supremum is over all probability matrices $H$ generated from $k$-block, $\rho$-bounded graphons. 

\citet[Corollary 3]{Klopp:2015} also studied the problem of estimating the graphon $W$. They proved that Equation (\ref{Kloppupper}) holds for any $k$-block, $\rho$-bounded graphon, $W$, with $k\le n$. They also exhibited the first lower bound (known to us) for graphon estimation using the $\delta_2$ metric. They proved that Equation (\ref{Klopplower}) holds for $\rho>0$ and $k\le n$.

The related problems of distinguishing a graphon with $k> 1$ from an Erd\"os-R\'enyi model with the same average degree (called the distinguishability problem) and reconstructing the communities of a given network (called the reconstruction problem) have also been widely studied. This problem is closely related to the problem of estimating $H$. Recent work by \citet{Mossel:2014} and \citet{Neeman:2014} establish conditions under which a $k$-block graphon is mutually contiguous to the Erd\"os-R\'enyi model with the same average degree. Contiguity is essentially a condition that implies that no test could ever definitely determine which of the two graphons a given sample came from. There is a large body of work on algorithmic and statistical problems in this area and we have only cited work that is directly relevant here. 


\section{Lower Bound for the $\delta_2$ Metric}

As mentioned earlier, the main technical contribution of this paper is lower bounding the $\delta_2$ metric by the more combinatorial $\doublehat{\delta_2}$ metric. In this section we will prove the inequality given in Lemma \ref{lowerbounddelta2}.

\begin{prop}\label{expectation}
Let $W, W'$ be $k$-block graphons with blocks $I_i=[\frac{i-1}{k}, \frac{i}{k})$ and $\pi:[0,1]\to[0,1]$ be a measure-preserving bijection. Then there exists a probability distribution $\mathbb{P}$ on $\mathcal{S}_k$ such that $$\|W_{\pi}-W'\|_2^2 = \underset{\sigma, \tau\sim\mathbb{P}}{\mathbb{E}}[\|W_{\sigma, \tau}-W'\|_2^2].$$
\end{prop}

\begin{proof}
Let $a_i = \frac{i-1}{k}$ and $p_{ij} = \mu(I_i\cap \pi^{-1}(I_j))$. Now, consider a $k\times k$ matrix $P$ with $P_{ij} = kp_{ij}$. Noting that $\sum_{j=1}^k p_{ij} = \mu(I_i) = 1/k$ and $\sum_{i=1}^k p_{ij} = \mu(\pi^{-1}(I_j)) = 1/k$, we can see that $P$ is doubly stochastic, that is, the rows and columns of $P$ sum to 1. Berkhoff's theorem states that any doubly stochastic matrix can be written as a convex combination of permutation matrices. So  $P=\sum_{\sigma\in\mathcal{S}_k}\mathbb{P}(\sigma)\sigma$ where $\sum_{\sigma\in\mathcal{S}_k}\mathbb{P}(\sigma)=1$. Therefore, we have a probability distribution $\mathbb{P}$ on $\mathcal{S}_k$ and $$\mathbb{P}(\sigma(i)=j) = \sum\{\mathbb{P}(\sigma) \:|\: \sigma(i)=j\} = P_{ij} = kp_{ij}.$$ 

Now, 
\begin{align*}
\mathbb{E}[\|W_{\sigma, \tau}-W'\|_2^2]&=\sum_{\sigma,\tau} \mathbb{P}(\sigma)\mathbb{P}(\tau)\sum_{i,j} \frac{1}{k^2}(W(a_{\sigma(i)}, a_{\tau(j)})-W'(a_i,a_j))^2 \\
&= \sum_{i,i',j,j'} \frac{1}{k^2}\mathbb{P}(\sigma(i)=i')\mathbb{P}(\tau(j)=j')(W(a_i, a_j)-W'(a_{i'}, a_{j'}))^2 \\
&= \sum_{i,i',j,j'} p_{ii'}p_{jj'}(W(a_i, a_j)-W'(a_{i'},a_{j'}))^2\\
&= \|W_{\pi}-W'\|_2^2.\qedhere
\end{align*}
\end{proof}

\begin{proof}[Proof of Lemma \ref{lowerbounddelta2}]\label{deltahathat}
Proposition \ref{expectation} implies that for all measure preserving bijections $\pi:[0,1]\to[0,1]$ and matrices $A$ and $B$ we have $$\|W[A]_{\pi}-W[B]\|_2\ge\inf_{\sigma, \tau\in\mathcal{S}_k}\|W[A]_{\sigma, \tau}-W[B]\|_2=\inf_{\sigma, \tau\in\mathcal{S}_k}\|A_{\sigma, \tau}-B\|_2 = \doublehat{\delta_2}(A, B).$$ Since this is true for any $\pi$, we have $\delta_2(W[A],W[B])\ge\doublehat{\delta_2}(A,B)$. 
\end{proof}

\section{Proof of Main Theorem}

To prove the main theorem we will use Fano's lemma to find a constant that lower bounds the probability that the estimation exceeds $\min\left(\rho, \sqrt{\frac{\rho k^2}{n^2}}\right)$, which then implies the appropriate lower bound on the expected $\delta_2$ error. To that end, we aim to find a large set, $T$, of $k$-block graphons whose KL-diameter and $\epsilon$-packing number with respect to $\delta_2$ with $\epsilon=\min\left(\rho, \sqrt{\frac{\rho k^2}{n^2}}\right)$ can be bounded. Our proof is inspired by that of Gao et al. 

Suppose $p, q$ are probability distributions on the same space. Then the Kullback-Leibler (KL) divergence of $p$ and $q$ is defined by $D(p\|q) = \int(\log\frac{dp}{dq})dp$. For a collection $T$ of  probability distributions, the KL diameter is defined by $$d_{KL}(T)=\sup_{p,q\in T} D(p\| q).$$ The following version of Fano's lemma is found in \cite{Yu:1997}.

\begin{lem}[Fano's Inequality.]
Let $(\Omega, d)$ be a metric space and $\{\mathbb{P}_{\theta} \: |\: \theta\in\Omega\}$ be a collection of probability measures. For any totally bounded $T\subset\Omega$ and $\epsilon>0$, $$\inf_{\hat{\theta}}\sup_{\theta\in\Omega}\mathbb{P}_{\theta}\left(d^2(\hat{\theta}(X), \theta)\ge \frac{\epsilon^2}{4}\right)\ge 1-\frac{d_{KL}(T)+1}{\log\mathcal{M}(\epsilon, T, d)}$$ 
where the infimum is over all estimators.
\end{lem}

The following lemma gives us a way to easily upper bound the KL divergence between the distributions induced by two different graphons. 

\begin{lem}\label{graphonKL}
For any graphons $\frac{1}{2}\le W,W'\le \frac{3}{4}$, we have $$D(G_n(W)\|G_n(W'))\le 8n^2\|W-W'\|_{2}^2.$$
\end{lem}

\begin{proof}
Let $T$ be a variable denoting the choice of labels, so $$\mathbb{P}_{G_n(W)}(G) = \int_{\ell\in[0,1]^n} \mathbb{P}_T(\ell)\mathbb{P}_{G_n(W)}(G|T=\ell)d\ell.$$ Now,
\begin{align*}
D(G_n(W)\|G_n(W')) &= \sum_{G\in G_n} \mathbb{P}_{G_n(W)}(G)\ln\left(\frac{\mathbb{P}_{G_n(W)}(G)}{\mathbb{P}_{G_n(W')}(G)}\right)\\
&\le \sum_{G\in G_n} \int_{\ell\in[1,0]^n} \mathbb{P}_T(\ell)\mathbb{P}_{G_n(W)}(G|T=\ell)\ln\left(\frac{\mathbb{P}_{G_n(W)}(G|T=\ell)}{\mathbb{P}_{G_n(W')}(G|T=\ell)}\right)d\ell\\
&= \int_{\ell\in[0,1]^n} \mathbb{P}_T(\ell)D(\mathbb{P}_{G_n(W)}(\cdot|T=\ell)\|\mathbb{P}_{G_n(W')}(\cdot|T=\ell)d\ell,
\end{align*}
where the inequality follows from the log-sum inequality. Now, the probability density function of T is the constant function 1 so it follows from \citet[Proposition 4.2]{Gao:2014} that 
\begin{align*}
D(G_n(W)\|G_n(W')) &\le 8\int_{\ell\in[0,1]^n}\sum_{i,j=1}^n (W({\ell_i},{\ell_j})- W'({\ell_i},{\ell_j}))^2d\ell\\
&=8\sum_{i,j=1}^n \int_{\ell\in[0,1]^n} (W(\ell_i,\ell_j)- W'(\ell_i,\ell_j))^2 d\ell\\
&\le 8n^2\int_{[0,1]^2} (W(x,y)- W'(x,y))^2dxdy\\
&= 8n^2\|W-W'\|_2^2
\end{align*}
\end{proof}

Recall that we are aiming to define a large set of $k$-block matrices whose KL diameter can be upper bounded and packing number with respect to $\delta_2$ (with $\epsilon=\min(\rho, \sqrt{\frac{\rho k^2}{n^2}})$) can be lower bounded. The following lemma shows that there exists a large set of matrices who are pairwise far in Hamming distance, even after every possible permutation of the rows and columns. We will use this in the proof of Theorem \ref{main} to define a large class of $k$-block graphons who are pairwise far in the $\doublehat{\delta_2}$ metric and hence the $\delta_2$ metric. This gives us a bound on packing number. 

\begin{lem}\label{largeset}
There exists a set $S$ of symmetric $k\times k$ binary matrices such that $|S| = 2^{\Omega(k^2)}$ and if $B, B'\in S$ and $\sigma, \tau\in\mathcal{S}_k$ then Ham$(B_{\sigma, \tau}, B')= \Omega(k^2)$.
\end{lem}

\begin{proof}
Consider two randomly chosen symmetric binary matrices $B, B'$ and permutations $\sigma$ and $\tau$. For $i\le j$, let $X_{ij}=1$ if $B_{\sigma(i), \tau(j)}=B'_{i,j}$ and 0 otherwise so $X_{ij}$ is a Bernoulli random variable with $\mathbb{E}[X_{ij}]=\frac{1}{2}$. Thus, by a Chernoff bound, $$\mathbb{P}\left(\text{Ham}(B_{\sigma, \tau}, B')<\frac{k^2}{6}\right)=\mathbb{P}\left(\sum_{i\le j}X_{ij}\le \frac{k^2}{6}\right)\le e^{\frac{-2\left(\frac{k^2}{6}-\frac{1}{2}\binom{k}{2}\right)^2}{\binom{k}{2}}}.$$
Therefore, for randomly chosen $B, B'$, $$\mathbb{P}\left(\exists\sigma, \tau \text{ s.t. Ham}(B_{\sigma, \tau}, B')<\frac{k^2}{6}\right)\le e^{\frac{-2\left(\frac{k^2}{6}-\frac{1}{2}\binom{k}{2}\right)^2}{\binom{k}{2}}}(k!)^2 = 2^{-\Omega(k^2)}.$$ 
For a constant $c>0$, consider the process that selects $2^{ck^2}$ binary matrices $\{B_i\}_{i}$ uniformly at random uniformly at random. The probability that all pairs are at Hamming distance at least $k^2/6$ is at least $1-2^{2ck^2}2^{-\Omega(k^2)}$. Selecting $c$ sufficiently small, we get that at least one such set $S$ exists. 
\end{proof} 

We are not aware of an explicit construction of a large family of matrices that are far apart in $\doublehat{\delta}_2$ metric; we leave such a construction as an open problem.

We now proceed to the proof of Theorem \ref{main}. We will use Lemma \ref{largeset} to define a set $T$ with packing number $2^{\Omega(k^2)}$. The elements of $T$ are all close in $\|\cdot\|_{\infty}$ norm, so using Lemma \ref{graphonKL} we get a bound on the KL diameter. We then directly apply these bounds to Fano's lemma.

\begin{thm}
For any positive integer $k\le n$ and $0< \rho\le 1$, 
$$\inf_{\hat{W}}\sup_{W}\underset{G\sim G_n(W)}{\mathbb{E}}\left[\delta_2(\hat{W}(G), W)\right]\ge \Omega\left(\min\left(\rho, \;\sqrt{\frac{\rho k^2}{n^2}}\right)\right).$$
where $\inf_{\hat{W}}$ is the infimum over all estimators $\hat{W}:G_n\to\mathcal{G}$ and $\sup_W$ is the supremum over all $k$-block, $\rho$-bounded graphons.
\end{thm}

\begin{proof}

Let $S$ be a set satisfying the conditions of Lemma \ref{largeset} and let $\eta = \min(1, \frac{k}{n\sqrt{\rho}})$. For $B\in S$, define $$Q_B = \rho\left[\frac{1}{2}\mathbf{1}+c\eta(2B-\mathbf{1})\right],$$ where $\mathbf{1}$ is the all 1's matrix and $c$ is some constant that we will choose later. That is, $(Q_B)_{ij}=\rho[\frac{1}{2}+c\eta]$ if $B_{ij}=1$ and $(Q_B)_{ij} = \rho[\frac{1}{2}-c\eta]$ if $B_{ij}=0$. Let $T= \left\{W[Q_B] \:|\: B\in S\right\}$ then using Lemma \ref{graphonKL} we conclude that for all $W, W'\in T$, we have $$D(G_n(W)\|G_n(W')) \le 8n^2(2c\rho\eta)^2 \le 32c^2k^2\rho$$ so $d_{KL}(T)\le O(c^2k^2\rho)$. 

Let $B, B'\in S$ and suppose $\sigma, \tau\in\mathcal{S}_k$. Then by construction, $$\|({W[Q_B]})_{\sigma, \tau}-W[Q_{B'}]\|_2^2 \ge \frac{1}{k^2}\text{Ham}(B_{\sigma, \tau}, B')(2\rho c\eta)^2 = \Omega(c^2\rho^2\eta^2).$$ Thus by Corollary \ref{deltahathat}, $$\delta_2(W[Q_B], W[Q_{B'}])\ge\doublehat{\delta_2}(W[Q_B], W[Q_{B'}])\ge \Omega(c\rho\eta).$$ Therefore, there exists $D>0$ such that if $\epsilon = D\rho c\eta = D\min\left(c\rho, \frac{ck\sqrt{\rho}}{n}\right)$, we have $\log\mathcal{M}(\epsilon, T, \delta_2) = \Omega(k^2)$. Then, Fano's lemma implies $$\inf_{\hat{W}}\sup_{W} \text{Pr}\left(\delta_2(\hat{W}, W)\ge \frac{\epsilon}{2}\right)\ge 1-\frac{O(c^2k^2\rho)+1}{\Omega(k^2)}.$$ We can choose $c$ small enough that the right hand side is larger than a fixed constant for all $k$ and $n$. Therefore, using Markov's inequality we have $$\inf_{\hat{W}}\sup_{W}\mathbb{E}\left[\delta_2(\hat{W}, W)\right]\ge \Omega\left(\epsilon\right) = \Omega\left(\min\left(\rho, \sqrt{\rho\frac{k^2}{n^2}}\right)\right).$$
\end{proof}

\begin{proof}[Proof of Proposition \ref{covering}]
During the course of the proof of Theorem \ref{main} we construct $2^{\Omega(k^2)}$ graphons in $\mathcal{W}_{\rho}$ that are pairwise at least $\Omega(\rho c\eta)$ apart in the $\delta_2$ distance for any $c>0$ such that $|c\eta|\le\frac{1}{2}$. Therefore, for some $C>0$, the $C\rho$-packing number of $\mathcal{W}_{\rho}$ is at least $2^{\Omega(k^2)}$.
\end{proof}

\section{Appendix}\label{extra}

We show here how to derive the lower bound in \eqref{Neemanlower} from the results of \citet{Neeman:2014}.

\begin{prop}\label{Neemanbit}
For any positive integer $2\le k\le n$ and $0< \rho\le 1$, 
$$\inf_{\hat{W}}\sup_{W}\underset{G\sim G_n(W)}{\mathbb{E}}[\delta_2(\hat{W}, W)]\ge\Omega\left(\min\left(\rho,  \sqrt{\frac{\rho \log (\min(k,\rho n+2))}{n}}\right)\right)$$
where $\inf_{\hat{W}}$ is the infimum over all estimators $\hat{W}:G_n\to\mathcal{G}$ and $\sup_W$ is the supremum over all $k$-block, $\rho$-bounded graphons.
\end{prop}

\begin{lem}\label{Neemanlem}
For any positive integer $k\le n$ and $0< \rho\le 1$, 
$$\inf_{\hat{W}}\sup_{W}\underset{G\sim G_n(W)}{\mathbb{E}}\left[\delta_2(\hat{W}(G), W)\right]\ge \Omega\left(\min\left(\frac{\rho}{\sqrt{k}}, \;\sqrt{\frac{\rho \log k}{n}}\right)\right).$$
where $\inf_{\hat{W}}$ is the infimum over all estimators $\hat{W}:G_n\to\mathcal{G}$ and $\sup_W$ is the supremum over all $k$-block, $\rho$-bounded graphons.
\end{lem}

\begin{proof} Let $\epsilon = \min\left(\sqrt{\frac{\rho k\log k}{n}}, \rho\right)$, $q=\frac{1}{2}\frac{k-1}{k^2}\frac{n\epsilon^2}{\log(k-1)}$ and $p=\epsilon+q$. Let $W_1$ be a $k$-block graphon with equally sized blocks and edge probabilities between nodes with the same label, $p$, and with different labels, $q$. Let $W_2$ be the Erd\"os-R\'enyi model with the same expected degree, $d$, as $W_1$. Note that $0\le p,q\le O(\rho)$ so $W_1, W_2\in\mathcal{W}_{O(\rho)}$ and $\delta_2\left(W_1, ER\right)\ge \Omega\left(\frac{|p-q|}{\sqrt{k}}\right)=\Omega\left(\frac{\epsilon}{\sqrt{k}}\right).$

Let $(\Omega_n, \mathcal{F}_n)$ be a sequence of measurable spaces, each equipt with two probability measures, $\mathbb{P}_n$ and $\mathbb{Q}_n$. We say $\mathbb{P}_n$ and $\mathbb{Q}_n$ are mutually contiguous if for any sequence of events $A_n$, we have $\lim_{n\to\infty}\mathbb{P}_n(A_n)\to0$ if and only if $\lim_{n\to\infty}\mathbb{Q}_n(A_n)\to0$. Let $\lambda = \frac{n(p-q)}{dk}$. A result of Neeman et al. \cite{Neeman:2014}, as presented by Banks et al. \cite{Banks:2016}, implies that $W_1$ and $W_2$ are mutually contiguous if 
$$\frac{d\lambda^2(k-1)}{2}\le \log(k-1).$$
A short calculation shows that this holds so we have that $W_1$ and $W_2$ are mutually contiguous. Choose $C>0$ such that $\delta_2(W_1, W_2) \ge C\frac{\epsilon}{\sqrt{k}}$. 

Suppose, for sake of contradiction, that there did exist $\hat{W}$ such that $$\sup_{W\in\mathcal{W}_{\rho}}\mathbb{E}_{G\sim G_n(W)}\left[\delta_2(\hat{W}(G), W)\right]\le o\left(\frac{\epsilon}{\sqrt{k}}\right)= o\left(\min\left(\frac{\rho}{\sqrt{k}}, \sqrt{\frac{\rho \log k}{n}}\right)\right).$$ 
Then due to the contiguity of $W_1$ and $W_2$,
$\lim_{n\to\infty}Pr_{G\sim G_n(W_1)}[\delta_2(\hat{W}(G), W_1)\ge \frac{C}{2}\frac{\epsilon}{\sqrt{k}}]\to 0$ and  $\lim_{n\to\infty}Pr_{G\sim G_n(W_1)}[\delta_2(\hat{W}(G), W_2)\ge \frac{C}{2}\frac{\epsilon}{\sqrt{k}}]\to 0$. Therefore, for large enough $n$, there exists a graph $G$ such that $\delta_2(\hat{W}(G), W_1)< \frac{C}{2}\frac{\epsilon}{\sqrt{k}}$ and $\delta_2(\hat{W}(G), W_2)< \frac{C}{2}\frac{\epsilon}{\sqrt{k}}$, which implies that $\delta_2(W_1, W_2)<C\frac{\epsilon}{\sqrt{k}}$, which is a contradiction.
\end{proof}

\begin{proof}[Proof of Proposition \ref{Neemanbit}]
Let $f(i) =\min\left(\frac{\rho}{\sqrt{k}}, \;\sqrt{\frac{\rho \log k}{n}}\right)$. Any $k$-block graphon is also an $i$-block graphon for any $i\le k$ so Lemma \ref{Neemanlem} implies $$\inf_{\hat{W}}\sup_{W}\underset{G\sim G_n(W)}{\mathbb{E}}\left[\delta_2(\hat{W}(G), W)\right]\ge \Omega\left(\max_{i\le k} f(i)\right)\ge \Omega\left(\min\left(\rho, \sqrt{\frac{\rho \log(\min(k, \rho n+2))}{n}}\right)\right).$$
If $\rho\le \frac{4}{n}$ then $$\Omega\left(\max_{i\le k} f(i)\right)\ge\Omega\left(\min\left(\rho, \sqrt{\frac{\rho}{n}}\right)\right)\ge \Omega\left(\min\left(\rho, \sqrt{\frac{\rho \log(\min(k, \rho n+2))}{n}}\right)\right).$$
If $\rho\ge \frac{k\log k}{n}$ then $$\Omega\left(\max_{i\le k} f(i)\right)=\Omega\left(\frac{\rho\log k}{n}\right)\ge \Omega \left(\sqrt{\frac{\rho \log(\min(k, \rho n+2))}{n}}\right).$$
If $\frac{4}{n}\le \rho\le \frac{k\log k}{n}$ then the functions $\frac{\rho}{\sqrt{i}}$ and $\sqrt{\frac{\rho\log k}{n}}$ intersect when $\rho n=i\log i$, which occurs for some $i$ such that $i\ge \sqrt{\rho n}$ so $$\Omega\left(\max_{i\le k} f(i)\right)\ge \Omega\left(\sqrt{\frac{\rho\log((\rho n)^{1/4})}{n}}\right)\ge \Omega \left(\sqrt{\frac{\rho \log(\min(k, \rho n+2))}{n}}\right)$$

\end{proof}


\bibliographystyle{abbrvnat}
\bibliography{bibliography,graphsprivacy,PrivateGraphs}

\end{document}